\documentclass[12pt]{article}

\usepackage{amssymb,amsmath,latexsym,amsthm}
\usepackage[english]{babel}
\usepackage[all]{xy}

\newtheorem{theorem}{Theorem}[section]
\newtheorem{lemma}{Lemma}[section]
\newtheorem{lem}{Lemma}[section]
\newtheorem{prop}{Proposition}[section]
\newtheorem{cor}{Corollary}[section]

\theoremstyle{definition}

\newtheorem*{ex}{Example}

\begin{document}

\title{Eichler orders, trees and Representation Fields}


\author{\sc Luis Arenas-Carmona}


\newcommand\Q{\mathbb Q}
\newcommand\alge{\mathfrak{A}}
\newcommand\Da{\mathfrak{D}}
\newcommand\Ea{\mathfrak{E}}
\newcommand\Ha{\mathfrak{H}}
\newcommand\oink{\mathcal O}
\newcommand\matrici{\mathbb{M}}
\newcommand\Txi{\lceil}
\newcommand\ad{\mathbb{A}}
\newcommand\enteri{\mathbb Z}
\newcommand\finitum{\mathbb{F}}
\newcommand\bbmatrix[4]{\left(\begin{array}{cc}#1&#2\\#3&#4\end{array}\right)}
\newcommand\lbmatrix[4]{\textnormal{\scriptsize{$\left(\begin{array}{cc}#1&#2\\#3&#4
\end{array}\right)$}}\normalsize}
\maketitle

\begin{abstract}
If $\Ha\subseteq\Da$ are two orders in a central simple algebra
$\alge$ with $\Da$ of maximal rank, the representation field
$F(\Da|\Ha)$ is a subfield of the spinor class field of the genus
of $\Da$ which determines the set of spinor genera of orders in
that genus representing the order $\Ha$. Previous work have
focused on two cases, maximal orders $\Da$ and commutative orders
$\Ha$. In this work, we show how to compute the representation
field $F(\Da|\Ha)$ when $\Da$ is the intersection of a finite
family of maximal orders, e.g., an Eichler order, and $\Ha$ is
arbitrary. Examples are provided.
\end{abstract}

\bigskip
\section{Introduction}

Let $K$ be a global field, and let $\alge$ be a quaternion algebra
over $K$. Let $X$ be an A-curve with field of functions $K$ as
defined in \S2 of \cite{abelianos}. For example, when $K$ is a
number field we can assume $X=\mathrm{Spec}(\oink_{K,S})$ for a
finite set $S=X^c$ containing the infinite places, but we also
include the case in which $X$ is an affine or projective curve
over a finite field, where in the latter case we set
$S=\emptyset$. In \cite{abelianos} we computed a representation
field that determines the set of spinor genera of maximal orders
that contain a conjugate of a given commutative order. In fact,
the definition of genera, spinor genera, spinor class fields, and
representation fields, given in \cite{abelianos} for maximal
order, has a straightforward extension for orders of maximal rank,
extending the usual definition when $K$ is a number field
\cite{spinor}. When strong approximation holds, spinor genera
coincide with conjugacy classes, just as in the number field case.
 The existence of a
representation field $F$ for $\Ha$ implies that the proportion of
conjugacy classes in the genus $\mathbb{O}=\mathrm{Gen}(\Da)$
representing $\Ha$ is $[F:K]^{-1}$. This fact was first studied by
Chevalley \cite{Chevalley} when $\alge$ is a matrix algebra of
arbitrary dimension, $\Da$ is a maximal order, and $\Ha$ is the
maximal order in a maximal subfield in $\alge$. In more recent
times, several authors have studied the problem in the more
restricted case of quaternion algebras. The following table
summarizes the main results in this area:

\footnotesize
\[
\begin{tabular}{ | c | c | c | c |c| }
  \hline
  Year & Ref. & $\Da$ & $\Ha$ & other \\
  \hline\hline
  1999 & \cite{FriedmannQ} & maximal & commutative &  \\  \hline
  2004 & \cite{Guo}& Eichler& commutative & \\   \hline
    2004 & \cite{Chan}& Eichler& commutative & Proves existence with no restriction on $\Da$ \\   \hline
  2008 & \cite{eichler} & maximal & Eichler  & Finds couterexample with $\dim_K\alge=9$  \\
  \hline
    2008 & \cite{Macla} & EOSFL & commutative & Considers optimal embeddings  \\  \hline
   2011 & \cite{lino1} & & commutative & Assumes some technical conditions on\\
&&&& the pair $(\Da,\Ha)$ \\ \hline
\end{tabular}
\]
\normalsize Here EOSFL means Eichler order of square free level.
In 2011, the representation field was computed for commutative
orders $\Ha$ in maximal orders $\Da$ of central simple algebras of
arbitrary dimension \cite{abelianos}. The commutativity condition
on $\Ha$ is only necessary in a technical step, and in fact the
method in \cite{abelianos} allows the computation of spinor class
fields for other interesting families of orders, like cyclic
orders \cite{cyclic}. However, the condition that $\Da$ is maximal
is essential in this computations and a generalization to
arbitrary orders of maximal rank seems unlikely at this point.

To fix ideas, we say that two $X$-orders $\Da$ and $\Da'$, of
maximal rank in $\alge$, are in the same genus, if
$\Da'_\wp=a_\wp\Da a_\wp^{-1}$ for some local element $a_\wp$ in
every completion $\alge_\wp^*$. Equivalently, $\Da$ and $\Da'$ are
in the same genus if $\Da'=a\Da a^{-1}$ for some adelic element
$a\in\alge_\ad^*$. Similarly, two $X$-orders $\Da$ and $\Da'$, of
maximal rank in $\alge$, are in the same spinor genus, if
$\Da'_\wp=b\Da b^{-1}$ for some local element
$b=rc\in\alge_\ad^*$, where $r\in\alge^*$ is a global element,
while $c$ is an element with trivial reduced norm. The spinor
class field $\Sigma=\Sigma(\mathbb{O})$ for a genus $\mathbb{O}$
of orders of maximal rank is defined as the class field
corresponding to $K^*H(\Da)$, where $\Da$ is an order in
$\mathbb{O}$, and $H(\Da)$ is the group of reduced norms of adelic
elements stabilizing $\Da$ by conjugation. This definition is
independent of the choice of $\Da\in\mathbb{O}$, since the reduced
norm is invariant under conjugation. The importance of this field
lies in the existence of a map
$$\rho:\mathbb{O}\times\mathbb{O}\rightarrow\mathrm{Gal}(\Sigma/K),$$
with the following properties:
\begin{enumerate}
\item $\Da$ and $\Da'$ are conjugate if and only if
$\rho(\Da,\Da')=\mathrm{Id}_\Sigma$, \item
$\rho(\Da,\Da'')=\rho(\Da,\Da')\rho(\Da',\Da'')\qquad\forall
(\Da,\Da',\Da'')\in\mathbb{O}^3$,
\end{enumerate}
When $\Ha$ is a suborder of $\Da$ (or some other order in the
genus of $\Da$) the representation field $F=F(\Da|\Ha)$ is the
class field of the set $K^*H(\Da|\Ha)$, where
$$H(\Da|\Ha)=\{N(a)|a\in\alge_\ad^*,\textnormal{ and }a\Ha
a^{-1}\subseteq\Da\},$$ if this set turns out to be a group. It
has the property that $\Ha$ embeds into an order
$\Da'\in\mathbb{O}$ if and only if $\rho(\Da,\Da')$ is trivial on
$F$. When $F(\Da|\Ha)$ is defined the number of spinor genera
representing $\Ha$ divides the total number of spinor genera. This
is not always the case in algebras of higher dimension
\cite{eichler}. The proof of all these facts is a word-by-word
transliteration of the proof for maximal orders (\cite{abelianos},
\S2).

The purpose of the current work is to give a formula for the
representation field $F(\Da|\Ha)$ whenever $\Da=\oink_X+I\Da_0$
for an Eichler order $\Da_0$ in a quaternion algebra $\alge$, and
an integral ideal, i.e., a $1$-dimensional lattice, $I$ in $K$.
Our result has no restriction on the sub-order $\Ha$:
\begin{theorem}\label{th11}
Let $\Da=\oink_X+I\Da_0$ in the preceding notations. Then the
spinor class field of $\Da$ is the largest field $\Sigma$
satisfying the following conditions:
\begin{enumerate}
\item $\Sigma/K$ is an extension of exponent $2$ unramified at the finite primes. \item $\Sigma/K$
splits completely at every place $\wp$ satisfying one of the following conditions:
\begin{enumerate} \item $\wp$ is non-archimedean and $\wp\notin X$, \item $\wp$ is archimedean and $\alge_\wp$ is a matrix algebra, \item $\wp\in X$ and $\alge_\wp$ is a division algebra,
 \item the level $d_\wp$ of $\Da_0$ at $\wp$ is odd.
\end{enumerate}\end{enumerate}
Furthermore, for any suborder $\Ha\subseteq\Da$, the
representation field is the largest subfield splitting completely
at every place $\wp$ where $\Ha$ embeds locally in
$\oink_X+I\Da_1$ for an Eichler order $\Da_1$ whose level at $\wp$
is strictly larger than $d_\wp$.
\end{theorem}
The description of the class group of $\Sigma$, in the specific
case of Eichler orders and without the language of spinor class
fields, appears already in Corollary III.5.7 in \cite{vigneras}.
Here we obtain this computation with no effort as a consequence of
the general computation of relative spinor images in terms of
branches obtained in \S3. The condition in the last statement of
Theorem \ref{th11} can easily be decided by the methods described
in \S4. As a consequence, we generalize the results for Eichler
orders in \cite{Guo} or \cite{Chan} in this manner (\S6).

 The importance of the orders considered here lays in the fact
 that they are the only orders that can be
written as the intersection of a family of maximal orders. In
fact, as a byproduct of our work on these orders, we prove in \S5
the following result:
\begin{theorem}\label{thm2}
For an order of maximal rank $\Da$, the following conditions are
equivalent:
\begin{enumerate}
\item $\Da=\oink_X+I\Da_0$ for an Eichler order $\Da_0$ and an
integral ideal $I$ such that $I_\wp=\oink_\wp$ whenever $\alge_\wp$
is a division algebra. \item $\Da$ is the intersection of a finite
family of maximal orders.\item $\Da$ is the intersection of three
maximal orders.
\end{enumerate}
\end{theorem}

Note that if $\Da$ is an arbitrary order of maximal rank and
$\Da'$ is the intersection of the maximal orders containing $\Da$,
we have $F(\Da'|\Ha)\subseteq F(\Da|\Ha)$, so at least an
effective lower bound for the representation field can be obtain
by the methods of the present paper for any order of maximal rank.
This can be used in some cases to prove selectivity.

\section{Trees and branches}

In all of this section, we let $K$ be a local field with ring of
integers $\oink_K$ and maximal ideal $m_K=\pi\oink_K$. Let $\mathfrak{T}$ be
the Bruhat-Tits tree of $PGL_2(K)$, i.e., the vertices of
$\mathfrak{T}$ are the maximal orders in $\matrici_2(K)$, while
two of them, $\Da_1$ and $\Da_2$ are joined by an edge if and only
if $[\Da_1:\Da_1\cap\Da_2]=[\oink_K:m_K]$ (\cite{trees}, \S II.1).
Recall that every maximal order has the form
$\Da_{\Lambda}=\{x\in\matrici_2(K)|x\Lambda\subseteq\Lambda\}$ for
some lattice $\Lambda\subseteq K^2$, and $\Da_\Lambda=\Da_{\Lambda'}$ if and only if $\Lambda'=\lambda\Lambda$ for some $\lambda\in K^*$. For any order $\Ha$ (of
arbitrary rank) in $\matrici_2(K)$, we define
$\Ha^{[s]}=\oink_K+\pi^s\Ha$.

\begin{prop}\label{ll1}
For every non-negative integers $s$ and $t$, the following
properties hold:
\begin{enumerate}\item If
$\Ha^{[t]}\subseteq\Ha_1^{[t]}$, for some integer $t\geq0$, then
$\Ha^{[s]}\subseteq\Ha_1^{[s]}$ for any integer $s\geq0$.
\item $(\Ha_1\cap\Ha_2)^{[s]}=\Ha_1^{[s]}\cap\Ha_2^{[s]}$. \item
$(\Ha^{[s]})^{[t]}=\Ha^{[s+t]}$.
\end{enumerate}\end{prop}

\begin{proof}
Property (3) and the case $t=0$ of the first statement are
trivial, and therefore
$(\Ha_1\cap\Ha_2)^{[s]}\subseteq\Ha_1^{[s]}\cap\Ha_2^{[s]}$. Next
we prove
$\Ha_1^{[s]}\cap\Ha_2^{[s]}\subseteq(\Ha_1\cap\Ha_2)^{[s]}$.
Assume $\lambda+\pi^s\rho=\mu+\pi^s\sigma$, where
$\lambda,\mu\in\oink_K$, $\rho\in\Ha_1$, and $\sigma\in\Ha_2$.
Then $\rho$ and $\sigma$ commute, whence
$\sigma-\rho=\pi^{-s}(\lambda-\mu)$ is integral over $\oink_K$. We
conclude that $\sigma-\rho\in\oink_K$, whence (2) follows. The
general proof of (1) is similar.
\end{proof}

For any order $\Ha\subseteq\matrici_2(K)$ and any integer $r\geq0$
we define the set
$$S_r(\Ha)=\{\Da\in V(\mathfrak{T})|\Ha\subseteq
\Da^{[r]}\},$$ and call it the $r$-branch of $\Ha$. Next result
follows easily from the definition:
\begin{prop}\label{basics} The following properties hold for any
order $\Ha\subseteq\matrici_2(K)$:
\begin{enumerate}
\item $S_0(\Ha)$ is non-empty. \item For every triple of integers
$(r,k,t)$ we have $S_{r+t}(\Ha^{[k+t]})=S_r(\Ha^{[k]})$. \item If
$\Ha\subseteq \Ha'$, then $S_r(\Ha)\supseteq S_r(\Ha')$ for any
integer $r$. \item If $\Ha'$ is the intersection of a family of
maximal orders, and if $S_0(\Ha)\supseteq S_0(\Ha')$, then
$\Ha\subseteq \Ha'$. \item $\Ha\subseteq\Da^{[r]}$ for some
Eichler order $\Da$ of level $d$ if and only if $S_r(\Ha)$
contains two vertices at distance $d$.
\end{enumerate}
\end{prop}

A set of vertices in a tree is said to be connected if for every
pair of its points it contains every vertex in the unique path
joining them. Let $\Da,\Da'\in S_r(\Ha)$. Then
$\Ha\subseteq\oink_K+\pi^r(\Da\cap\Da')$. Now recall that the
Eichler order $\Da\cap\Da'$ is contained in every order in the
path joining $\Da$ and $\Da'$. Next result follows:

\begin{prop}
For any order $\Ha$ and any integer $r$, the branch $S_r(\Ha)$ is
connected.\end{prop}

\begin{prop}\label{lt3}
For any order $\Ha$ and any integer $t$, the branch
$S_0(\Ha^{[t]})$ contains exactly the maximal orders at a distance
$\leq t$ of some maximal order in $S_0(\Ha)$.\end{prop}

\begin{proof}
It suffices to prove the case $t=1$ and use (3) in Proposition
\ref{ll1}. Now observe that two maximal orders $\Da$ and $\Da_1$
are neighbors if and only if, in some basis $\{v,w\}$ they
corresponds to the lattices $\Lambda=\oink_Kv+\oink_Kw$ and
$\Lambda_1=\oink_Kv+\pi\oink_Kw$ respectively. The fact that
$\Da_1\in S_0(\Ha)$ means that for every element $h\in \Ha$ we
have $h\Lambda_1\subseteq\Lambda_1$. Note that $\pi
h\Lambda_1\subseteq\Lambda_1$, and $h(\pi\Lambda)\subseteq
h\Lambda_1\subseteq\Lambda_1\subseteq\Lambda$ for every $h\in\Ha$.
Then $\Ha^{[1]}\Lambda_1\subseteq\Lambda_1$ and
$\Ha^{[1]}\Lambda\subseteq\Lambda$, whence sufficiency follows. On
the other hand, assume that $\Da\in S_0(\Ha^{[1]})$, and set
$\Da=\Da_\Lambda$.
 Then for every $h\in\Ha$, we have
$\pi h\in\Ha^{[1]}$, whence $\pi h\Lambda=\Lambda$, and therefore
the $\Ha$-invariant lattice $\Lambda'=\Ha\Lambda$ is contained in
$\pi^{-1}\Lambda$. There are three possibilities:
\begin{enumerate}
\item If $\Lambda'=\Lambda$, then $\Da\in S_0(\Ha)$.
 \item If $[\Ha':\Ha]=|\oink_K/m_K|$, then the maximal order $\Da'$
 corresponding to $\Lambda'$ is a neighbor of $\Da$ and $\Da'\in
 S_0(\Ha)$.\item If $\Lambda'=\pi^{-1}\Lambda$, then
 $\pi^{-1}\Lambda$,
 and therefore also $\Lambda$, are $\Ha$-invariant. This is a
 contradiction, so this case cannot hold.
\end{enumerate}
The result follows.
\end{proof}

\begin{cor} Let $\Da$ be a maximal order. Then $S_0(\Da^{[t]})$ is the set of orders at a distance
at most $t$ from $\Da$.
\end{cor}

In particular, a maximal order contains $\Da^{[t]}$ if and only if
it is located at a distance not bigger than $t$ in the graph.
In fact, a stronger statement is true.

\begin{lem}\label{int1} $\Da^{[t]}$ is the intersection of all orders at a distance
at most $t$ from $\Da$.
\end{lem}

\begin{proof}
Let $\Da'$ be the intersection of all maximal orders at a distance
at most $t$ from $\Da$.
From the preceding corollary, $\Da^{[t]}\subseteq\Da'$. We prove the converse.
Let $u\in\Da'$, i.e., $u$ is contained in every maximal order at a distance
at most $t$ from $\Da$. Let $\Da_1$ and $\Da_2$ be two orders at distance $t$
from $\Da$, such that the shorter path between $\Da_1$ and $\Da_2$ passes through $\Da$.
In some choice of coordinates, the orders $\Da_1$, $\Da$, and $\Da_2$ are respectively:
$$\Da_1=\bbmatrix {\oink_K}{\pi^t\oink_K}{\oink_K}{\oink_K},\quad
\Da=\bbmatrix {\oink_K}{\oink_K}{\oink_K}{\oink_K},\quad
\Da_2=\bbmatrix {\oink_K}{\oink_K}{\pi^t\oink_K}{\oink_K}.$$ It
follows that $u\equiv\lbmatrix a00b\ (\mathrm{mod}\ \pi^t)$ for
some $a,b\in\oink_K$. Since conjugation by the element
$h=\lbmatrix 1101$ leaves invariant $\Da$, and therefore also the
set of maximal orders at a distance no bigger that $t$ from $\Da$,
the element $huh^{-1}\equiv\lbmatrix a{b-a}0b\ (\mathrm{mod}\
\pi^t)$ also belongs to $\Da'$. We conclude that $b\equiv a\
(\mathrm{mod}\ \pi^t)$, and therefore $u\in\Da^{[t]}$.
\end{proof}

A maximal order $\Da\in S_0(\Ha)$ such that every maximal order at
a distance at most $t$ from $\Da$ belongs to $S_0(\Ha)$ is said to
be $t$-deep in $S_0(\Ha)$.

\begin{cor} $S_r(\Ha)$ is the set of $r$-deep maximal orders in $S_0(\Ha)$.
\end{cor}

We conclude that in order to compute $S_r(\Ha^{[t]})$, it suffices
to compute $S_0(\Ha)$.

\section{relative spinor images}

In all of this section $K$ is a local field with maximal order
$\oink_K$, uniformizing parameter $\pi$, and absolute value
$x\mapsto|x|_K$. Let $\alge$ be a split quaternion $K$-algebra,
i.e., $\alge\cong\matrici_2(K)$, and let $N:\alge^*\rightarrow
K^*$ be the reduced norm. For any pair of orders $\Ha\subseteq\Da$
in $\alge$ define the local relative spinor image
$H(\Da|\Ha)\subseteq K^*$ by
$$H(\Da|\Ha)=\{N(u)|u\Ha u^{-1}\subseteq\Da\},$$
and let $H(\Da)=H(\Da|\Da)$. Recall that
$H(\Da|\Ha)=H(\Da)H(\Da|\Ha)H(\Ha)$ \cite{spinor}. Note that for
any order $\Da$ and any invertible element $u$ in $\alge$, we have
$u\Da^{[r]} u^{-1}=(u\Da u^{-1})^{[r]}$ for any non-negative
integer $r$, and the correspondence $\Da\mapsto\Da^{[r]}$ is
injective and preserves inclusions. It follows that
$H(\Da^{[r]}|\Ha^{[r]})=H(\Da|\Ha)$. In particular
$H(\Da^{[r]})=H(\Da)$.

Let $\Da=\Da_1\cap\Da_2$ be an Eichler order of level $d$. In
other words, $[\Da_i:\Da]=|\pi|_K^d$ for $i=1,2$. Note that $d$
is the distance between $\Da_1$ and $\Da_2$ in the bruhat-Tits
tree $\mathfrak{T}$ of $\mathrm{PGL}_2(K)$. Let $\Ha$ be a
suborder of $\Da^{[r]}$. Let $S_r(\Ha)$ be as defined in the
preceding section. It follows that $S_r(\Ha)$ has two points,
namely $\Da_1$ and $\Da_2$, at distance $d$. Note that
$H(\Da^{[r]})=H(\Da)\supseteq\oink_K^*K^{*2}$, since the Eichler
order $\Da$ contains a conjugate of the diagonal matrix $\mathrm{diag}(1,u)$ for every
unit $u\in\oink_K^*$. In particular, $H(\Da^{[r]})$ is either
$\oink_K^*K^{*2}$ or $K^*$.

\begin{lemma}\label{ll2}
$H(\Da^{[r]}|\Ha)=K^*$ if and only if there exist another pair
$(\Da_3,\Da_4)$ in $S_r(\Ha)\times S_r(\Ha)$ also at distance $d$
such that the distance between $\Da_1$ and $\Da_3$ is odd.
\end{lemma}

\begin{proof}
Assume the condition is satisfied, and take $\sigma\in GL_2(K)$
such that $\sigma\Da_1\sigma^{-1}=\Da_3$ and
$\sigma\Da_2\sigma^{-1}=\Da_4$. Then
$\Ha\subseteq(\Da^{[r]}_3\cap\Da^{[r]}_4)=\sigma\Da^{[r]}\sigma^{-1}$.
It follows that $\sigma$ is a local generator for $\Da^{[r]}|\Ha$.
Since the distance between $\Da_1$ and $\Da_3$ is odd and
$\sigma\Da_1\sigma^{-1}=\Da_3$, the reduced norm $n(\sigma)$ has
odd valuation (Corollary to Prop. 1 in \S II.1.2 of \cite{trees}).
 Assume now that $H(\Da^{[r]}|\Ha)=K^*$, so
there must exists a generator $\sigma$ with reduced norm of odd
valuation. Then let $\Da_3=\sigma\Da_1\sigma^{-1}$ and
$\Da_4=\sigma\Da_2\sigma^{-1}$. Since $\sigma$ is a generator, we
must have $\Ha\subseteq\sigma\Da^{[r]}
\sigma^{-1}=(\Da_3\cap\Da_4)^{[r]}$, while the fact that the
reduced norm of $\sigma$ has odd valuation means that the distance
from $\Da_1$ to $\Da_3$ is also odd (Corollary to Prop. 1 in \S
II.1.2 of \cite{trees}).
\end{proof}

\begin{cor}\label{cor1} If the level of $\Da=\Da_1\cap\Da_2$ is odd then
$H(\Da^{[r]}|\Ha)=K^*$.
\end{cor}

\begin{proof}
It suffices to define $(\Da_3,\Da_4)=(\Da_2,\Da_1)$.
\end{proof}

\begin{cor}\label{cor32} If the level of $\Da$ is smaller than the diameter of $S_r(\Ha)$, then
$H(\Da^{[r]}|\Ha)=K^*$.
\end{cor}

\begin{proof}
Let $\Ea_0-\Ea_1-\cdots-\Ea_{d+1}$ a path of length $d+1$ in
$S_r(\Ha)$. Then either $(\Da_3,\Da_4)=(\Ea_0,\Ea_d)$ or
$(\Da_3,\Da_4)=(\Ea_1,\Ea_{d+1})$ satisfies the hypotheses of the
previous lemma.
\end{proof}

If $\Da$ is maximal, so that $d=0$, then $H(\Da^{[r]}|\Ha)=K^*$ as
soon as $\Ha$ is contained in a second order of the form
$\Da_1^{[r]}$ for $\Da_1$ maximal. On the other hand, if
$S_r(\Ha)$ contains exactly one point, then the condition of Lemma
\ref{ll2} cannot be satisfied. Next result follows:

\begin{cor} If $\Da$ is maximal, then $H(\Da^{[r]}|\Ha)=K^*$ if and
only if $\Ha\subseteq\Da_0^{[r]}$ for a maximal order $\Da_0$
implies $\Da_0=\Da$. In particular, $H(\Da|\Ha)=K^*$ if and only
if $\Ha$ is contained in a unique maximal order.
\end{cor}

If $\mathbb{H}$ is the finite algebra defined in \S 3 of
\cite{abelianos}, then in the quaternionic case the irreducible
representations of $\mathbb{H}$ have dimensions $1$ or $2$, and in
the latter case this representation is unique. This can happen
only if $\mathbb{H}$ contains the unique quadratic extension of
the finite field $\finitum_\wp$. From the preceding corollary and
Lemma 3.3 in \cite{abelianos}, next result follows:

\begin{cor}\label{c34} If $\Da$ is maximal, then $\Ha$ is contained in a unique maximal
order if and only if it contains the maximal order of an
unramified extension.
\end{cor}

\begin{prop}\label{prop31} Let $d$ be an even integer.
If $d$ is the diameter of $S_r(\Ha)$, and $\Da$ has level $d$,
then $H(\Da_r|\Ha)=\oink_K^*K^{*2}$.\end{prop}

This result follows from next lemma:

\begin{lemma}
Let $T$ be a finite tree with even diameter $d$. Let $(x_1,x_2)$
and $(x_3,x_4)$ be two pairs of vertices at distance $d$. Then the
distance from $x_1$ to $x_3$ is even.
\end{lemma}

\begin{proof}
Denote by $\rho$ the usual distance in the graph $T$. Consider the
path joining $(x_1,x_2)$ and the path joining $(x_3,x_4)$. If no
edge of these paths is a common edge, there must be a unique
(possibly empty) minimal path joining two of their vertices, say
$y$ and $y'$, as the figure show:\[ \fbox{ \xygraph{
!{<0cm,0cm>;<.8cm,0cm>:<0cm,.8cm>::} !{(1,0) }*+{\bullet^{y}}="a"
!{(3,0)}*+{{}^{y'}\bullet}="b" !{(0,1) }*+{{}^{x_1}\bullet}="c"
!{(0,-1) }*+{{}_{x_2}\bullet}="d" !{(4,1)}*+{\bullet^{x_3}}="e"
!{(4,-1)}*+{\bullet_{x_4}}="f" "a"-@{.}"b" "a"-"c" "a"-"d" "b"-"e"
"b"-"f" } }
\] Note that
$$\rho(x_1,x_4)+\rho(x_3,x_2)+\rho(x_1,x_3)+\rho(x_2,x_4)=4\Big[d+\rho(y,y')\Big],$$
but $d$ is the diameter, whence $y=y'$ and every distance on the
left is equal to $d$. It follows that
$$\rho(x_1,y)=\frac12\Big[\rho(x_1,x_4)+\rho(x_1,x_3)-\rho(x_3,x_4)\Big]=d/2.$$
the same argument holds for all other extremes. Now assume that
the path joining $y$ and $y'$ is common to both, the path joining $x_1$
and $x_2$, and the path joining $x_3$ and $x_4$ like in next picture:\[ \fbox{ \xygraph{
!{<0cm,0cm>;<.8cm,0cm>:<0cm,.8cm>::} !{(1,0) }*+{\bullet^{y}}="a"
!{(3,0)}*+{{}^{y'}\bullet}="b" !{(0,1) }*+{{}^{x_1}\bullet}="c"
!{(0,-1) }*+{{}_{x_3}\bullet}="d" !{(4,1)}*+{\bullet^{x_2}}="e"
!{(4,-1)}*+{\bullet_{x_4}}="f" "a"-@{=}"b" "a"-"c" "a"-"d" "b"-"e"
"b"-"f" } }
\] This includes the possibility that $y$ or $y'$ coincide with
one of the endpoints. Then $\rho(x_3,x_2)\leq d=\rho(x_1,x_2)$,
whence $\rho(x_3,y)\leq \rho(x_1,y)$, and by symmetry,
$\rho(x_3,y)=\rho(x_1,y)$. The result follows.
\end{proof}

By setting $\Ha=\Da^{[r]}$, we obtain from Corollary \ref{cor1}
and the preceding proposition:

\begin{cor}\label{cor35}
$H(\Da^{[r]})$ is $\oink_K^*K^{*2}$ if $d$ is even and $K^*$
otherwise.
\end{cor}

\paragraph{Proof of Theorem \ref{th11}}
Recall that  The spinor class field of an order $\Da$ is the class
field of the group $K^*H(\Da)\subseteq J_K$. Write
$H(\Da)=\prod_\wp H_\wp(\Da)$ in terms of the local spinor images
defined in this section. Now the first statement is a consequence
of the following observations:
\begin{enumerate}\item $J_K^2\subseteq H(\Da)$ since scalar
matrices normalize any order. \item If $\wp\notin X$, then
$H_\wp(\Da)$ contains $N(\alge^*_\wp)=K^*_\wp$, unless $\wp$ is a
real place that is ramified for $\alge$, in which case
$N(\alge^*_\wp)=K_\wp^+$. \item Similarly, at places $\wp$ where
$\alge_\wp$ is a division algebra we have
$H_\wp(\Da)=H_\wp(\Da_0)=N(\alge_\wp^*)=K_\wp^*$, since $\Da_0$ is
the unique maximal order. \item If $\wp$ is a place where the
level of $\Da_0$ is odd, then $H_\wp(\Da)=K^*_\wp$ by Corollary
\ref{cor35}. By the same result $H_\wp(\Da)=\oink_\wp^*K^{*2}_\wp$
at the remaining places.
\end{enumerate} The second statement is proven analogously by
using Corollary \ref{cor32} and Proposition \ref{prop31}.\qed

\section{Branches of commutative orders and computations}

We note that if an order $\Ha$ is expressed in terms of a set of
generators $\{a_1,\dots,a_s\}$, then a maximal order contains
$\Ha$ if and only if it contains all the generators. It follows
that $$S_r(\Ha)=\bigcap_{i=1}^sS_r\Big(\oink_K[a_i]\Big).$$ It
sufices therefore to compute $S_r(\Ha)$ for commutative orders.
Recall that  we can always assume $r=0$ in these computations.

\begin{lem}
If $L$ is a semisimple commutative subalgebra of $\alge$,
and $\Omega$ is an order in $L$, then $\Omega=\oink_L^{[t]}$ for some non-negative integer $t$.
\end{lem}

\begin{proof}
Note that for any element $a\in\alge$ that is integral over
$\oink_K$, an element of the form $a+\lambda$ with $\lambda\in K$
is integral over $\oink_K$ if and only if $\lambda\in\oink_K$. We
conclude that $\Omega$ is completely determined by its image in
the abelian group $L/K$. The result follows since every order is
contained in a maximal order and the $K$-vector space $L/K$ is
one-dimensional.\end{proof}

\begin{prop}
Let $L$ be a quadratic extension of $K$. Then
$S_0(\oink_L)=S_0(\Da)$ where $\Da$ is a maximal order if the
extension $L/K$ is unramified and an Eichler order of level $1$
otherwise. When $L$ is isomorphic to $K\times K$, then
$S_0(\oink_L)$ is a maximal path in the tree.
\end{prop}

\begin{proof}
If $\oink_L$ is the maximal order in a quadratic field extension
$L/K$, then the maximal orders containing $\oink_L$ are in
correspondence with the $\oink_L$-invariant lattices in the
(unique) two dimensional representation of $L$ as a $K$-algebra.
By identifying $K^2$ with $L$ as vector spaces, we obtain that the
maximal orders containing $\oink_L$ are in correspondence with the
fractional ideals on $L$ modulo $K^*$-multiplication. In other
words, $\oink_L$ is contained in a unique maximal order if $L/K$
is unramified and exactly two (necessarily neighbors) if $L/K$ is
ramified. The last statement is similar. If $L\cong K\times K$,
the fractional ideals are of the form
$\pi^s\oink_K\times\pi^t\oink_K$. It is readily seen that the
corresponding orders lie in a maximal path.
\end{proof}

\begin{cor}
If $\Omega$ is a commutative order contained in a field, then any
branch $S_r(\Omega)$ is the set of orders at distance not
exceeding $t$ from either a vertex or two neighboring vertices,
for some non-negative integer $t$. If $\Omega$ is a commutative
order contained in an algebra isomorphic to $K\times K$, then any
branch $S_r(\Omega)$ is the set of orders at distance not
exceeding $t$ from some maximal path in the tree, for some
non-negative integer $t$.
\end{cor}

\begin{prop}
Let $\Omega=\oink_K[a]$ where $a\neq0$ and $a^2=0$. Let $\Lambda$
be a lattice of the form $\Lambda=\oink_Kv+\oink_Kw$, where $av=w$
and $aw=0$. Then if $\Da_i$ is the maximal order corresponding to
the lattice $\pi^i\oink_Kv+\oink_Kw$, then $S_0(\Omega)$ is the
set of maximal orders at a distance at most $i$ from $\Da_i$ for
some $i\geq0$.
\end{prop}

\begin{proof}
Without loss of generality we can assume $(v,w)=(e_1,e_2)$ is the
canonical basis. In particular
$$a=\left(\begin{array}{cc}0&0\\1&0\end{array}\right),\textnormal{ and }
\Da_i=\left(\begin{array}{cc}\oink_K&\pi^i\oink_K\\
\pi^{-i}\oink_K&\oink_K\end{array}\right).$$ It follows that $a$
is contained in each of the orders $\Da_i$ for $i\geq0$. Now, let
$h=\lbmatrix\pi001$. Then $h^{-1}\Da_ih=\Da_{i-1},$ while
$h^{-1}\Omega h=\Omega^{[1]}$. It follows that
$h^{-1}S_0(\Omega)h=S_0(\Omega^{[1]})$, and by iteration
$h^{-t}S_0(\Omega)h^t=S_0(\Omega^{[t]})$ for every $t\geq1$. Note
now that $\Da_1$ is the unique neighbor of $\Da_0$ containing
$\Omega$, since $a$ does not stabilize a lattice of the form
$\oink_K^2+\pi^{-1}u\oink_K$ with $u\in\oink_K^2$ unless
$u\in\oink_Ke_2+\pi\oink_K^2$ . Then $\Da_0$ has valency $1$ in
$S_0(\Omega)$, and therefore, the vertices $\Da\in \mathfrak{T}$
satisfying $\rho(\Da_0,\Da)\leq\rho(\Da_1,\Da)$ are in
$S_0(\Omega^{[t]})$ if and only if $\rho(\Da_0,\Da)\leq t$. We
conclude that the vertices $\Da\in \mathfrak{T}$ satisfying
$\rho(\Da_i,\Da)\leq\rho(\Da_{i+1},\Da)$ are in $S_0(\Omega)$ if
and only if  $\rho(\Da_i,\Da)\leq i$, and the result follows.
\end{proof}

\begin{cor}
Let $\Omega=\oink_K[a]$ where $a^2=0$. Then there is a conjugation
taking $S_r(\Omega)$ to $S_0(\Omega)$ for every non-negative
integer $r$.
\end{cor}

\begin{ex}
Let $\Ha\subseteq\Da_0$ be the order generated by
$a_1=\lbmatrix00{\pi^2}0$ and $a_2=\lbmatrix0{\pi^3}00$.
Then $S_0(\Ha)$ is the intersection of the branches
$$S_0\Big(\oink_K[a_1]\Big)=\{\Da|\rho(\Da,\Da_i)\leq i+2\textnormal{ for
some } i\geq-2\},$$
$$S_0\Big(\oink_K[a_2]\Big)=\{\Da|\rho(\Da,\Da_i)\leq3-i\textnormal{ for
some } i\leq3\}.$$ It follows that $S_0(\Ha)$ is the set of orders
at a distance not bigger than 2 from either $\Da_0$ or $\Da_1$,
and therefore $S_0(\Ha)=S_0\Big(\Da_0^{[2]}\cap\Da_1^{[2]}\Big)$.
\end{ex}

\begin{ex}\label{ex42}
Assume $K$ is a non-dyadic local field, and let $\Ha$ be an order
generated by elements $i,j$ satisfying $i^2=j^2=1$ and $ij=-ji$.
It follows from the results in this section that $\oink_K[i]$ and
$\oink_K[j]$ are maximal paths. Note that
$\mathbb{H}=\Ha/\pi_K\Ha$ is a matrix algebra, whence by Corollary
\ref{c34} $\Ha$ is contained in a unique maximal order. In fact,
it is not hard to prove that $\Ha$ is a maximal order. We conclude
that the paths $S_0(\oink_K[i])$ and $S_0(\oink_K[j])$ intersect
in a unique point.
\end{ex}

\section{Admisible shapes for branches}

Although the results in \S 3 can be obtained with no mention to
the specific shape of the branch $S_r(\Ha)$, it turns out that
there is only a very restricted set of possible branches. In order
to prove this we need some preparation.

\begin{lem}\label{lt4}
For any order $\Ha$, the branches $S_r(\Ha)$ satisfies the
following properties:
\begin{enumerate}
\item If $\Da\in S_r(\Ha)$ has two neighbors in $S_r(\Ha)$, one of
which is in $S_{r+1}(\Ha)$, then $\Da\in S_{r+1}(\Ha)$. \item If
$\Da\in S_0(\Ha)$, and there are three paths of length $r$ in
$S_0(\Ha)$ starting from $\Da$ with no common edge, as in the
following picture, then $\Da\in S_r(\Ha)$.
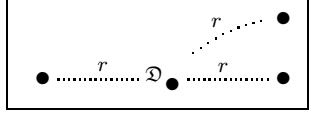
\[ \fbox{ \xygraph{
!{<0cm,0cm>;<.8cm,0cm>:<0cm,.8cm>::} !{(1,0) }*+{\bullet}="a"
!{(3,0)}*+{{}^{\Da}\bullet}="b" !{(5,0) }*+{\bullet}="c" !{(5,1)
}*+{\bullet}="d" "a"-@{.}^r"b" "b"-@{.}^r"c"
"b"-@`{p+(-1,0)}@{.}^r"d"} }
\]
\end{enumerate}
\end{lem}

\begin{proof}
Note that any property of the form $\Phi(X)=[(Y\subseteq
X)\implies(Z\subseteq X)]$ is stable under intersections. Recall
now that $\Da\in S_r(\Ha)$ if and only if the closed $r$-ball
$B[\Da;r]\subseteq\mathfrak{T}$ is contained in $S_0(\Ha)$, so
both properties in the lemma are of the given type. Now the result
follows by a straightforward case-by-case checking on the branches
described in \S 4.
\end{proof}

\begin{lem}\label{lt5}
For any order $\Ha$, the following properties hold:
\begin{enumerate}
\item If $S_{r+1}(\Ha)$ is not empty, then any maximal order in
$S_r(\Ha)$ has a neighbor in $S_{r+1}(\Ha)$. \item If
$S_{r+1}(\Ha)$ is empty, then $S_r(\Ha)$ is a (possibly infinite)
path.
\end{enumerate}
\end{lem}

\begin{proof}
Assume $S_{r+1}(\Ha)\neq\emptyset$. Let $\Da\in S_r(\Ha)$ and let
$\Da'\in S_{r+1}(\Ha)\subseteq S_r(\Ha)$. Since $S_r(\Ha)$ is
connected, there is a path $\Da'=\Ea_0-\cdots-\Ea_n=\Da$ in
$S_r(\Ha)$. Successive applications of the first part of the
previous lemma prove that $\Ea_0,\dots,\Ea_{n-1}$ are all in
$S_{r+1}(\Ha)$. First statement follows.

Assume now that $S_r(\Ha)$ is not a path. Then there exists a
vertex $\Da\in S_r(\Ha)$ with three neighbors also in $S_r(\Ha)$.
extending all of these lines by a path of length $r$ heading
in the opposite direction, as the picture shows, we obtain three
segments of length $r+1$ starting from $\Da$.
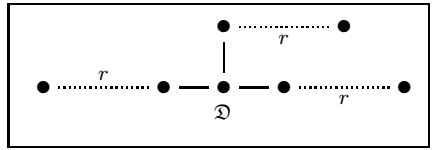
\[ \fbox{ \xygraph{
!{<0cm,0cm>;<.8cm,0cm>:<0cm,.8cm>::} !{(1,0) }*+{\bullet}="a"
!{(4,0)}*+{\bullet}="b" !{(4,-.5)}*+{{}^\Da}="h" !{(7,0)
}*+{\bullet}="c" !{(6,1) }*+{\bullet}="d" !{(3,0) }*+{\bullet}="e"
!{(5,0)}*+{\bullet}="f" !{(4,1) }*+{\bullet}="g" "a"-@{.}^r"e"
"c"-@{.}^r"f" "d"-@{.}^r"g" "b"-"f" "b"-"g" "b"-"e"}}
\]
It follows from the second part of the previous lemma that $\Da\in
S_{r+1}(\Ha)$ and the result follows.
\end{proof}

We call a subset $S$ of $\mathfrak{T}$ an $r$-thick path if it
consist of all vertices at a distance of at most $r$ from some,
finite or infinite, path. Next result is immediate from the
preceding corollary by an obvious induction.

\begin{prop}
If an order $\Ha$ satisfies $S_r(\Ha)\neq\emptyset$ and
$S_{r+1}(\Ha)=\emptyset$, then $S_k(\Ha)$ is an $(r-k)$-thick path
for every $0\leq k \leq r$.
\end{prop}

\begin{cor}\label{cor51}
If an order $\Ha$ is contained in finitely many maximal orders,
then there exists an Eichler order $\Da$ and a non-negative
integer $t$ such that $S_r(\Ha)=S_r(\Da^{[t]})$ for every
non-negative integer $r$. In particular, this is the case whenever
$\Ha$ has maximal rank.
\end{cor}

\begin{prop}
If an order $\Ha$ satisfies $S_r(\Ha)\neq\emptyset$ for every $r$,
then $\Ha=\oink_K[a]$ for some nilpotent element $a$.
\end{prop}

\begin{proof}
For every element $b\in\Ha$, we have $\oink_K[b]=\oink_K[a]$ for
some nilpotent element $a$ by the computations in \S4. Now, if
$\Ha$ contain two linearly independent nilpotent elements, the
explicit description of the branch $S_0\big(\oink_K[a]\big)$ plus
the fact that the stabilizers of ends in the Bruhat-Tits tree are
the borel subgroups of $\mathrm{GL}_2(K)$ (see \cite{trees}, \S
II.1.3) show that the half infinite path used to construct
$S_0\big(\oink_K[a]\big)$ and $S_0\big(\oink_K[a']\big)$ cannot
have the same end, and therefore both branches intersect in a
finite set. The result follows.
\end{proof}

\begin{lem}
An order $\Ha$ is the intersection of a finite family of maximal
orders if and only if $\Ha=\Da_0^{[r]}$ for some Eichler order
$\Da_0$ and some non-negative integer $r\geq0$.
\end{lem}

\begin{proof}
First we prove that if $\Ha$ is the intersection of a finite
family of maximal orders, so is $\Ha^{[t]}$ for every $t$. In
fact, if $\Ha=\bigcap_{\Da\in\Phi}\Da$, then
$\Ha^{[t]}=\bigcap_{\Da\in\Phi}\Da^{[t]}$ by Proposition
\ref{ll1}. By Lemma \ref{int1}, every order on the right of this
identity is the intersection of a family of maximal orders. It
follows that every order of the form $\Ha=\Da_0^{[r]}$, where
$\Da_0$ is an Eichler order, is the intersection of a family of
maximal orders. On the other hand, if $\Ha$ is the intersection of
a finite family of maximal orders, then $\Ha=\Da_0^{[r]}$ by
Corollary \ref{cor51} and property (4) in Proposition
\ref{basics}.
\end{proof}

\paragraph{Proof of Theorem \ref{thm2}}
Note that all properties are local, and trivial at places
ramifying $\alge$. We proceed locally. It is clear that (3)
implies (2), and (2) implies (1) from the preceding lemma. Let
$\Ha=\Da^{[r]}$ for an Eichler order $\Da$ of level $d$, say
$\Da=\Da_1\cap\Da_2$, where $\Da_1$ and $\Da_2$ are two maximal
orders at distance $d$. Let $\Da_3$, $\Da_4$, and $\Da_5$ be
maximal orders located as the picture shows:
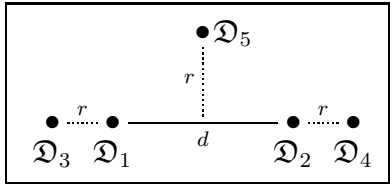
\[ \fbox{ \xygraph{
!{<0cm,0cm>;<.8cm,0cm>:<0cm,.8cm>::} !{(0,1) }*+{\bullet}="a"
!{(1,1)}*+{\bullet}="b" !{(4,1)}*+{\bullet}="c" !{(5,1)
}*+{\bullet}="d" !{(2.5,1) }*+{{}}="e" !{(2.5,2.5)
}*+{\bullet}="f" !{(0,0.5) }*+{\Da_3}="a1"
!{(1,0.5)}*+{\Da_1}="b1" !{(4,0.5)}*+{\Da_2}="c1" !{(5,0.5)
}*+{\Da_4}="d1" !{(3,2.5) }*+{\Da_5}="f1" "a"-@{.}^r"b"
"c"-@{.}^r"d" "e"-@{.}^r"f" "b"-_d"c"}}
\]
In other words, for instance, the path from $\Da_5$ to $\Da_3$
passes trough $\Da_1$, and $\rho(\Da_3,\Da_1)=r$.  Let
$\Ha'=\Da_3\cap\Da_4\cap\Da_5$ and let $S=S_0(\Ha')$. In
particular $S$ is a branch containing $\Da_3$, $\Da_4$, and
$\Da_5$. It follows from (1) in Lemma \ref{lt4} that there is an
$r$-deep vertex in the path joining $\Da_1$ and $\Da_2$. Now
reasoning as in the proof of Lemma \ref{lt5}, we see that both
$\Da_1$ and $\Da_2$, as much as every vertex in between, are
$r$-deep in $S$. We conclude that $S\supseteq S_0(\Ha)$, whence
$S=S_0(\Ha)$. Now $\Ha=\Ha'$ by (4) of Proposition \ref{basics},
and the result follows.\qed

\begin{ex}
Let $\Ha=\oink_K[\pi^ri,\pi^sj]$ where $i$ and $j$ are as in the
example at the end of \S4. Then $S_0(\Ha)$ contains every vertex
whose distance to the line $S_0(\oink_K[i])$ is no larger than $r$
and whose distance to the line $S_0(\oink_K[j])$ is no larger than
$s$. Let $\Da$ be the unique maximal order containing $i$ and $j$.
Then $S_0(\Ha)$ contains the two vertices in $S_0(\oink_K[i])$ at
a distance $s$ from $\Da$ and the two vertices in
$S_0(\oink_K[j])$ at a distance $r$ from $\Da$. In particular, if
$r<s$, it follows as in the preceding proof that
$S_0(\Ha)=S_0(\Ea^{[r]})$ for some Eichler order $\Ea$ of level
$2(s-r)$.
\end{ex}

\section{Examples}

Let $\Ha$ be a suborder of a global order $\Da=\oink_X+I\Da'$,
where $\Da'$ is an Eichler order. Assume first $K\Ha=L$ is a
field. Note that the global spinor image $H(\Da|\Ha)=\prod_\wp
H_\wp(\Da|\Ha)$ contains the group $N_{L/K}(J_L)K^*$ corresponding
to $L$, whence the representation field
 $F(\Da|\Ha)$ is a subfield of $L$. It follows that $F(\Da|\Ha)$
equals $K^*$ unless the local spinor image $H_\wp(\Da|\Ha)$ equals
$N_{L_\wp/K_\wp}(L_\wp^*)$ at every place $\wp$. Since the spinor
class field $\Sigma=\Sigma_{\Da}=\Sigma_{\Da'}$ is unramified,
split completely at every place where $\alge$ ramifies, and at
every place where the level of $\Da'$ is odd, we conclude that
$F(\Da|\Ha)=K^*$ unless $H_\wp(\Da|\Ha)=\oink_\wp^*K_\wp^{*2}$ at
every local place where $L/K$ is inert. Note that for local
unramified field extension $L_\wp/K_\wp$, the branch
$S_0(\Ha_\wp)$ of the order $\Ha_\wp=\oink_{L_\wp}^{[t]}$ is the
set of vertices at a distance not exceeding $t$ from the unique
maximal order containing $\oink_{L_\wp}$. It follows that the
diameter of $S_r(\Ha_\wp)$ is $2(t-r)$. On the other hand
$S_r(\Ha_\wp)$ is infinite when $L=K\Ha$ is not a field at $\wp$,
so that $H_\wp(\Da|\Ha)=K_\wp^*$ in this case. Next result
follows, generalizing the results in \cite{Chan} or \cite{Guo}:
\begin{prop} In the preceding notations, When $L=K\Ha$ is a
quadratic extension of $K$, we have $F(\Da|\Ha)=K$ unless the
following conditions are satisfied:
\begin{enumerate}
\item $L$ is contained in the spinor class field $\Sigma$. \item
For every place $\wp$ inert for $L/K$, we have
$\Ha_\wp=\oink_{L_\wp}^{\left[v_\wp(I)+\frac {d_\wp}2\right]}$
where $v_\wp$ denotes valuation at $\wp$, while $d_\wp$ is the
level at $\wp$ of the Eichler order $\Da'$. \end{enumerate} In the
latter case $F(\Da|\Ha)=L$. On the other hand, if $L=K\Ha$ is not
a field, then $F(\Da|\Ha)=K$.
\end{prop}

%
%

Assume now that $K\Ha$ is three-dimensional. This implies that
$\alge=\matrici_2(K)$, and $K\Ha$ is conjugated to the ring of
matrices of the form $\lbmatrix ab0c$. In particular, $\Ha$ is
contained in a conjugate of the order $\Ha_0=\lbmatrix
{\oink_X}{\oink_X}{0}{\oink_X}$. Since the branch
$S_0(\Ha_{0\wp})$ is infinite at every place $\wp$, next result
follows:
\begin{prop} If $\Ha$ is an order of rank $3$, then $F(\Da|\Ha)=K$.
\end{prop}

The situation is different for orders of rank $4$. In fact, by
choosing suborders of the type $\Ha=\oink_X+J\Da'$ where $J$ is an
ideal contained in $I$, we obtain that $F(\Da|\Ha)$ splits at
every place dividing $J/I$. Next result follows:
\begin{prop} For every field $F$ with $K\subseteq F\subseteq\Sigma$ there exists an order
 $\Ha$ of rank $4$, such that $F(\Da|\Ha)=F$.
\end{prop}

Note that the diameter of $S_0(\Da_0^{[t]})$, where $\Da_0$ is an
Eichler order of level $d$ is $\rho=d+2t$, and therefore the
diameter of $S_r(\Da_0^{[t]})$ is $d+2(t-r)$. Next result follows.

\begin{prop} Let $\Da=\oink_X+I\Da_1$ and $\Ha=\oink_X+J\Da_2$,
where $\Da_1$ and $\Da_2$ are Eichler orders. Let $l_\wp(\Da_i)$
denote the level at $\wp$ of the Eichler order $\Da_i$, then $\Ha$
embeds into an order in $\mathrm{gen}(\Da)$ if and only if at
every place $\wp$ the following inequalities hold:
$$v_\wp(J)\geq v_\wp(I),\qquad l_\wp(\Da_2)+2v_\wp(J)\geq
l_\wp(\Da_1)+2v_\wp(I).$$ Furthermore, if $\Ha\subseteq\Da$, then
$F(\Da|\Ha)$ is the largest subfield of $\Sigma(\Da)$ splitting at
every place where the second inequality is strict.
\end{prop}


\begin{thebibliography}{xx}



\bibitem{spinor}
{\sc L.E. Arenas-Carmona}, \textit{Applications of spinor class
fields: embeddings of orders and quaternionic lattices}, Ann.
Inst. Fourier \textbf{53} (2003), 2021--2038.

\bibitem{eichler}
 {\scshape Luis Arenas-Carmona}. Relative spinor class
fields: A counterexample, \textit{Archiv. Math.} \textbf{91}
(2008), 486-491.


\bibitem{abelianos}
{\sc L.E. Arenas-Carmona}, \textit{Spinor class fields for
commutative orders}, to appear in Ann. Inst. Fourier.
arXiv:1104.1809v1 [math.NT].

\bibitem{cyclic}
{\sc L.E. Arenas-Carmona}, \textit{Representation fields for
cyclic orders}, Submitted.

\bibitem{Chan}
{\sc W.K. Chan} and {\sc F. Xu}, \textit{On representations of
spinor genera}, Compositio Math. \textbf{140.2} (2004), 287-300.

\bibitem{Chevalley}
{\sc C. Chevalley}, \textit{L'arithm\'etique sur les alg\`ebres de
matrices}, Herman, Paris, 1936.


\bibitem{FriedmannQ}
{\sc T. Chinburg} and {\sc E. Friedman}, \textit{An embedding
theorem for quaternion algebras}, J. London Math. Soc.
\textbf{60.2} (1999), 33-44.


\bibitem{Guo}
{\sc X. Guo} and {\sc H. Qin}, \textit{An embedding theorem for
Eichler orders}, J. Number Theory \textbf{107.2} (2004), 207-214.

\bibitem{lino1}
{\sc B. Linowitz}, \textit{Selectivity in quaternion algebras},
Preprint.

\bibitem{Macla}
{\sc C. Maclachlan}, \textit{Optimal embeddings in quaternion
algebras}, J. Number Theory \textbf{128} (2008), 2852-2860.

\bibitem{trees}
{\scshape J.-P. Serre},  \textit{Trees}, Springer Verlag, Berlin,
1980.

\bibitem{vigneras}
{\scshape M.-F. Vigneras}, \textit{Arithm\'etique des alg\`ebres
de Quaternions}, Springer Verlag, Berlin, 1980.



\end{thebibliography}
\end{document}